\documentclass[reqno,centertags, 12pt]{amsart}

\usepackage{amssymb,amsmath,amsfonts,amssymb}
-\marginparwidth \oddsidemargin -9mm \evensidemargin \oddsidemargin 
\usepackage{latexsym}
\advance\hoffset by 5mm

\def\@abssec#1{\vspace{.05in}\footnotesize \parindent .2in
{\bf #1. }\ignorespaces} 
\newtheorem{theorem}{Theorem}[section]

\newtheorem{proposition}[theorem]{Proposition}
\newtheorem{corollary}[theorem]{Corollary}

\def \Rm {\mathbb R}
\def \reals {\mathbb R}

\def\cS{\mathcal S}

\allowdisplaybreaks \numberwithin{equation}{section}

\title[Biomixing by chemotaxis]{Biomixing by chemotaxis and efficiency of biological reactions: the critical reaction case}
\author{Alexander Kiselev}
\thanks{Department of
Mathematics, University of Wisconsin, Madison, WI 53706, USA; email: kiselev@math.wisc.edu}
\author{Lenya Ryzhik}
\thanks{Department of
Mathematics, Stanford University, Stanford, CA 94305; email: ryzhik@math.stanford.edu}

\dedicatory{To Professor Peter Constantin}

\begin{document}


\begin{abstract}
Many phenomena in biology involve both reactions and chemotaxis.
 These processes can clearly influence each other, and chemotaxis can play an important role in sustaining and speeding up the reaction. In continuation of our work \cite{KR10}, we consider a
model with a single density function involving diffusion, advection, chemotaxis, and absorbing reaction. The model is motivated, in particular,
by the studies of coral broadcast spawning, where experimental observations of the
efficiency of fertilization rates significantly exceed the data obtained from numerical models that do not take chemotaxis (attraction of sperm gametes by a chemical secreted by egg gametes) into account. We consider the
case of the weakly coupled quadratic reaction term, which is the most natural from the
biological point of view and was left open in \cite{KR10}. The result is that similarly to \cite{KR10}, the chemotaxis plays a crucial role
in ensuring efficiency of reaction. However, mathematically, the picture is quite different in the quadratic reaction case and is more subtle. The reaction is now complete even in the absence of chemotaxis, but
the timescales are very different. Without chemotaxis, the reaction is very slow, especially for the weak reaction coupling. With chemotaxis, the timescale and efficiency of reaction are independent of the coupling
parameter.
\end{abstract}

\maketitle

\section{Introduction}\label{intro}

Our goal in this paper is to study the effect chemotactic attraction may have on reproduction processes in biology. A particular motivation for this study comes from the phenomenon of broadcast spawning. Broadcast spawning
is a fertilization strategy used by various benthic invertebrates (sea urchins, anemones, corals) whereby males and females release sperm and egg gametes into the surrounding flow. The gametes are positively buoyant, and
rise to the surface of the ocean. The sperm and egg are initially separated by the ambient water, and effective mixing is necessary for successful fertilization. The fertilized gametes form larva, which is negatively buoyant
and tries to attach to the bottom of the ocean floor to start a new colony. For the coral spawning problem, field measurements of the fertilization rates are rarely below $5$\%, and are often as high as $90$\%
\cite{Eck,Lasker,Penn,Yund}. On the other hand, numerical simulations based on the turbulent eddy diffusivity \cite{DShi} predict fertilization rates of less than $1$\% due to the strong dilution of gametes. The turbulent
eddy diffusivity approach involves two scalars that react and diffuse with the effective diffusivity taking the presence of the flow into account. 
instantaneous details of the advective transport not captured by eddy diffusivity approach. It is well known, however, that the geometric structure of the fluid flow lost in the turbulent diffusivity approach can be
important for improving the reaction rate (in the physical and engineering literature see  \cite{Ottino,Ro,Y}; in the mathematical literature see  \cite{MS,CKOR,KR,FKR,KZ} for further references). Recent work of Crimaldi,
Hartford, Cadwell and Weiss \cite{Weiss1,Weiss2} employed a more sophisticated model, taking into account the instantaneous details of the advective transport not captured by the eddy diffusivity approach. These papers
showed that vortex stirring can generally enhance the reaction rate, perhaps accounting for some of the discrepancy between the numerical simulations and experiment.

However, there is also experimental evidence that chemotaxis plays a role in coral fertilization: eggs release a chemical that attracts sperm
\cite{Colletal1,Colletal2,Miller1,Miller2}. Mathematically, chemotaxis has been extensively studied in the context of modeling mold and bacterial colonies. Since the original work
of Patlak \cite{Patlak} and Keller-Segel \cite{KS1,KS2} where the first PDE model of chemotaxis was introduced, there has been an enormous amount of effort devoted to the possible blow up and regularity of solutions, as well
as the asymptotic behavior and other properties (see \cite{Pert} for further references). However, we are not aware of any rigorous or even computational work on the effects of chemotaxis for improved efficiency of
biological reactions.

In this paper, we continue the study of a simple single partial differential equation modeling the fertilization process that we initiated in \cite{KR10}. The equation is given by
\begin{equation}\label{chemo}
\partial_t \rho+u\cdot\nabla\rho =
\Delta \rho+ \chi \nabla\cdot (\rho \nabla[(\Delta)^{-1}\rho])  -  \epsilon \rho^q, \,\,\,\rho(x,0)=\rho_0(x),~~x\in\Rm^d.
\end{equation}
Here, in the simplest approximation, we consider just one density, $\rho(x,t) \geq 0,$ corresponding to the assumption that  the densities of sperm and egg gametes are identical. The vector field $u$ in \eqref{chemo} models
the ambient ocean flow, is divergence free, regular and prescribed, independent of $\rho.$ The second term on the right is the standard chemotactic term, in the same form as it appears in the (simplified) Keller-Segel
equation (see \cite{Pert}). This term describes the tendency of $\rho(x,t)$ to move along the gradient of the chemical whose distribution is equal to $-\Delta^{-1}\rho.$ This is an approximation to the full Keller-Segel
system based on the assumption of chemical diffusion being much faster than diffusion of gamete densities. The term $(-\epsilon\rho^q)$ models the reaction (fertilization). The parameter $\epsilon$ regulates the strength of
the fertilization process. The value of $\epsilon$ is small due to the fact that an egg gets fertilized only if a sperm attaches to a certain limited area on its surface (estimated to be about 1\% of the total egg surface
in, for example, sea urchins eggs \cite{Vogel}). We do not account for the product of the reaction -- fertilized eggs -- which drop out of the process. We are interested in  the behavior of
\[
m_0(t)=\int_{\Rm^d}\rho(x,t)dx,
\]
which is the total fraction of the unfertilized eggs by time $t$. It is easy to see that $m_0(t)$ is monotone decreasing. High efficiency fertilization corresponds to $m_0(t)$ becoming small with time, as almost all egg
gametes are fertilized.

The case $q>2$ in \eqref{chemo} has been considered in \cite{KR10}. In this case, we proved that if chemotaxis is not present, there exists a constant $\mu_0(\epsilon,q,d,\rho_0)$ such that $m_0(t) \geq \mu_0$ for all $t.$
Moreover, if we fix everything except let $\epsilon \rightarrow 0,$ $\mu_0$ converges to $m_0(0).$ Hence, in the absence of chemotaxis weak reaction coupling leads to a
very weak reaction completion rate. On the other hand, we
showed that if $d=2$ and chemotaxis is present, then $m_0(t) \rightarrow \nu(\chi,q,\rho_0,u)$ as $t \rightarrow \infty.$ Here, $\nu$ tends to zero if we increase $\chi$ keeping everything else fixed. Moreover, $\nu$
and the time scale of the convergence do not depend on $\epsilon,$ showing drastic difference in behavior compared to the pure reaction-diffusion-advection case.

The case $d=q=2$ in \eqref{chemo} is the most natural one - it corresponds to the reaction proportional to the product of egg and sperm densities. It turns out that the situation in this case is more subtle. Without chemotaxis, it is not necessarily true that $m_0(t)$ is bounded below from zero for all times by a fixed positive constant. Instead, $m_0(t) \rightarrow 0$ as $t \rightarrow 0,$ but very slowly.
To formulate our results, let us define $F(z)= z/\log((ez)^{-1}).$
\begin{theorem}\label{freeub}[Upper bound in the absence of chemotaxis]
Suppose that $\rho(x,t)$ satisfies
\begin{equation}\label{fcrit}
\partial_t \rho +u \cdot \nabla \rho = \Delta \rho - \epsilon \rho^2, \,\,\,\rho(x,0)=\rho_0(x)>0,
\end{equation}
where $u(x,t) \in C^\infty(\Rm^2 \times [0,\infty))$ is divergence free, and its stream function $H$ is globally bounded.
Let the initial data $\rho_0$ be such that
\[
\int_{\reals^2} \rho_0(x) e^{\alpha |x|}\,dx < \infty,
\]
for
 $0<\alpha \leq \alpha_0.$ Then $\|\rho(\cdot,t)\|_{L^1} \rightarrow 0$ as $t \rightarrow \infty.$
Moreover, there exists $t_0\ge 0$ that depends only on $\alpha_0$ so that for all times $t \geq t_0$ we have
\begin{equation}\label{Fbound}
F(\|\rho(\cdot,t)\|_{L^1}) \leq \frac{F(\|\rho(\cdot,t_0)\|_{L^1})}{1+\epsilon C(\rho_0) F(\|\rho(\cdot,t_0)\|_{L^1})\log (t/t_0)}.
\end{equation}
\end{theorem}
\noindent\it Remarks. \rm
1. The origin of the condition on stream function will be clear from the proof. An additional constant drift can be accommodated
without any problem. Hence, the theorem holds, for example, for every smooth periodic flow $u.$ \\
2. In particular, the bound \eqref{Fbound} implies that $\|\rho(\cdot,t)\|_{L^1} \lesssim C\log\log t/\log t$ for large $t.$ \\
3. The condition on the decay of $\rho_0$ can be relaxed at the expense of making the bound in \eqref{Fbound} slightly weaker.
For example, it will be clear from the proof that for $\rho_0 \in \cS$ (Schwartz class), we have that for $t \geq 1$ and every $\sigma>0,$
\begin{equation}\label{mombound}
 \|\rho(\cdot,t)\|_{L^1} \leq  \frac{C(\sigma,\rho_0)}{\left(1+\epsilon  \log t \right)^{1-\sigma}}.
\end{equation}
The bound \eqref{Fbound} is pretty close to being sharp, as the next theorem shows.
\begin{theorem}\label{lowerbound} [A lower bound in the absence of chemotaxis]
Suppose that $\rho(x,t)$ satisfies \eqref{fcrit}
where $u(x,t) \in C^\infty(\Rm^2 \times [0,\infty))$ is divergence free. Suppose that the initial data $\rho_0 >0 \in \cS.$
Then there exists $0<t_0<\infty$ depending on $\rho_0$ and $\epsilon$ such that for $t \geq t_0$ we have
\begin{equation}\label{lbkey}
\|\rho(\cdot,t)\|_{L^1} \geq \frac{\|\rho(\cdot,t_0)\|_{L^1}}{1+C\epsilon \|\rho(\cdot, t_0)\|_{L^1} \log (t/t_0) \exp(\epsilon\|\rho(\cdot,t_0)\|_{L^1})}.
\end{equation}
Moreover, there exists $c_0>0$ so that if  $\epsilon \|\rho_0\|_{L^1}\le c_0,$ then
\begin{equation}\label{simplerlb}
\|\rho(\cdot,t)\|_{L^1} \geq \frac{\|\rho_0\|_{L^1}(1-\epsilon C \|\rho_0\|_{L^1})}{1+C\epsilon \|\rho_0\|_{L^1} \log (t/t_0)}
\end{equation}
for $t \geq t_0 \equiv C\|\rho_0\|_{L^1}^2/\|\rho_0\|_{L^2}^2,$ where $C$ is a universal constant.
\end{theorem}
\noindent
Hence the estimate of $\sim (1+\epsilon \log t)^{-1}$ decay rate of $L^1$ norm of $\rho$ for large $t$ is almost sharp.

In the case when chemotaxis is present, we prove essentially the same result as in \cite{KR10}.

\begin{theorem}\label{chemore}[Estimates with chemotaxis]
Let $d=2,$ and suppose that $u \in C^\infty(\Rm^d \times [0,\infty))$ is divergence free. Assume that $q=2$ and $\rho(x,t)$ solves \eqref{chemo} with $\rho_0 \geq 0 \in \cS.$ Then

a. If $u=0,$ 
then for every $\tau>0,$ we have
\begin{equation}\label{dec1a}
\|\rho(\cdot,\tau)\|_{L^1} \leq \frac{2}{\chi}\left(1+\sqrt{1+\frac{\chi m_2}{4\tau}}\right).
\end{equation}

b. If $u \ne 0,$ then $\lim_{t \rightarrow \infty}\|\rho(\cdot,t)\|_{L^1} \leq C(u,m_2)\chi^{-2/3}.$ Moreover, for $0 \leq \tau \leq \chi^{1/3}$ we have
\begin{equation}\label{dec1b}
\|\rho(\cdot,\tau)\|_{L^1} \leq C(u,m_2)(\chi\tau)^{-1/2}.
\end{equation}
\end{theorem}

As in \cite{KR10}, the key feature of Theorem~\ref{chemore} is that all bounds are independent of $\epsilon.$
Hence, chemotaxis provides a significant improvement over \eqref{simplerlb} for reasonable time scales.
The key ingredient in the proof of of this theorem in \cite{KR10} was global regularity of solutions to
\eqref{chemo} with $d=2,$ $q>2.$ It is well known that solutions of Keller-Segel equation can form singularities in finite time, but in our situation the reaction term has regularizing effect. However the proof in
\cite{KR10} relied on global in time control of $L^\infty$ norm to prove global regularity, and this argument needed
the condition $q>2$ for reaction to offset the chemotactic nonlinearity. Here, we will prove global regularity for the case $d=q=2$ using a different approach. Notice that the regularity properties of solutions to Keller-Segel type equations with reaction terms have been studied earlier in \cite{TeWi,Wink1,Wink2}, but in a different setting.

\section{The reaction-diffusion-advection case}\label{rda}

In this section, we prove Theorems~\ref{freeub} and \ref{lowerbound}.
Let us begin with Theorem~\ref{freeub}. In all arguments below $C$ will denote universal constants that may change from line to line in the proof.
Consider the advection-diffusion equation in $\reals^2$
\begin{equation}\label{eqad}
\partial_t \phi +u \cdot \nabla\phi = \Delta \phi, \,\,\,\phi(x,0)=\phi_0(x).
\end{equation}
Suppose $u$ is smooth and $u = \nabla^\perp H,$ so that $H$ is a stream function corresponding to $u.$
Let us denote $p(x,y,t)$ the fundamental solution of \eqref{eqad}:
\[
\phi(x,t) = \int_{\reals^2} p(x,y,t)\phi_0(y)\,dy.
\]
Recall the following well known bound on the
fundamental solution of \eqref{eqad}.
\begin{theorem}[Norris \cite{Norr}]\label{norristhm}
Let $\phi(x,t)$ satisfy \eqref{eqad}, and assume that the stream function $H(x)$ is bounded. Then
\begin{equation}\label{norrb}
p(x,y,t) \leq Ct^{-1}\exp(-|x-y|^2/Ct),
\end{equation}
where $C$ depends only on $\|H\|_{L^\infty}.$
\end{theorem}
The reference \cite{Norr} contains more general results, but we stated the minimum that we need here.
We have
\begin{corollary}\label{norrcor}
Suppose $\phi_0(x)>0$ is such that
\[
\int_{\reals^2} e^{\alpha|x|}\phi_0(x)\,dx < \infty,
\]
for some $0<\alpha \leq \alpha_0.$
Then under conditions of Theorem~\ref{norristhm} we have
\begin{equation}\label{momcon23}
\int_{\reals^2} e^{\alpha|x|}\phi(x,t)\,dx \leq C e^{C\alpha^2 t} \int_{\reals^2} e^{\alpha|x|}\phi_0(x)\,dx.
\end{equation}
\end{corollary}
\begin{proof}
A direct computation shows that
\[ \int_{\reals^2} e^{\alpha x_1} \phi(x,t)\,dx \leq C e^{C\alpha^2 t} \int_{\reals^2} \phi_0(y) e^{\alpha y_1}\,dy. \]
Combining this and similar computations with $-x_1$ and $\pm x_2$ in the exponent leads to \eqref{momcon23}.
\end{proof}
Now suppose that $\rho(x,t)$ solves
\begin{equation}\label{rhoeq1}
\partial_t \rho +u \cdot \nabla \rho = \Delta \rho - \epsilon \rho^2,
\end{equation}
and
\[
\int \rho_0(x) e^{\alpha |x|}\,dx < \infty,
\]
for every $\alpha \leq \alpha_0.$
\begin{proposition}\label{l2l1}
There exists $t_0(\alpha_0)$ such that if $t \geq t_0,$ then
\begin{equation}
\|\rho(\cdot,t)\|_{L^2}^2 \geq C(\rho_0) \frac{\|\rho(\cdot,t)\|_{L^1}^2}{t \log (\|\rho(\cdot,t)\|_{L^1}^{-1})}.
\end{equation}
\end{proposition}
\begin{proof}
Consider $\phi(x,t)$ solving \eqref{eqad} with initial data $\rho_0(x)>0.$ Then by comparison principle, $\rho(x,t) \leq \phi(x,t),$
and therefore by Corollary~\ref{norrcor} we have
\begin{equation}\label{rho345} \int_{\reals^2} e^{\alpha|x|}\rho(x,t)\,dx \leq Ce^{C\alpha^2 t} \int_{\reals^2} \rho_0(y) e^{\alpha |y|}\,dy \end{equation}
for every $\alpha \leq \alpha_0,$ $t \geq 0.$
We would like to find $R(t)$ such that \begin{equation}\label{simple} \int_{B_{R(t)}} \rho(x,t)\,dx \geq \frac12 \int_{\reals^2} \rho(x,t)\,dx, \end{equation}
where $B_{R(t)}$ is a ball of radius $R(t).$ Observe that by \eqref{rho345},
\[ \int_{B^c_{R(t)}} \rho(x,t)\,dx \leq C e^{C\alpha^2 t - \alpha R(t)} \int_{\reals^2} \rho_0(y)e^{\alpha|y|}\,dy. \]
Thus choosing $R(t)$ so that
\[ R(t) \geq C\alpha t + \alpha^{-1} \log (\|\rho(\cdot,t)\|_{L^1}^{-1}) + C(\rho_0,\alpha) \]
will ensure \eqref{simple}. Choosing
\[
\alpha = t^{-1/2} \sqrt{\log(\|\rho(\cdot, t)\|^{-1})}
\]
optimizes the estimate
provided that $t \geq t_0(\alpha_0)$ so that $\alpha \leq \alpha_0.$  Then, for
\[
R(t) = C(\rho_0)\sqrt{t \log \|\rho(\cdot,t)\|_{L^1}^{-1}},
\]
we have
\[ \int_{B_{R(t)}} \rho^2 \,dx \geq \left( \int_{B_{R(t)}} \rho\,dx \right)^2 \frac{1}{|B_{R(t)}|} \geq \frac{C(\rho_0) \|\rho(\cdot,t)\|_{L^1}^2}{t \log(\|\rho(\cdot,t)\|_{L^1}^{-1})}. \]
\end{proof}
Now we are ready to prove Theorem~\ref{freeub}.
\begin{proof}[Proof of Theorem~\ref{freeub}]
From Proposition~\ref{l2l1}, we know that for all $t \geq t_0,$
\begin{equation}\label{nest23}
\partial_t \|\rho(\cdot,t)\|_{L^1} =\epsilon\|\rho(\cdot,t)\|_{L^2}^2 \leq -\epsilon C(\rho_0) t^{-1} \|\rho(\cdot, t)\|_{L^1}^2/\log(\|\rho(\cdot,t)\|_{L^1}^{-1}). \end{equation}
Let us denote $F(z) = z/(\log (ez)^{-1}).$ Then \eqref{nest23} implies
\begin{equation}\label{final1}
F(\|\rho(\cdot,t)\|_{L^1}) \leq \frac{F(\|\rho(\cdot,t_0)\|_{L^1})}{1+\epsilon C(\rho_0) F(\|\rho(\cdot,t_0)\|_{L^1}) \log t/t_0}.
\end{equation}
\end{proof}

Next, we prove Theorem~\ref{lowerbound}.
\begin{proof}[Proof of Theorem~\ref{lowerbound}]
Suppose $\rho(x,t)$ satisfies \eqref{rhoeq1} with $\rho_0>0 \in \cS,$ and $\phi(x,t)$ solves \eqref{eqad} with the same initial data. Then
\[ \partial_t \|\phi(\cdot,t)\|_{L^2}^2 = -\|\nabla \phi\|_{L^2}^2 \leq -\frac{\|\phi\|_{L^2}^4}{\|\phi\|_{L^1}^2} = -\frac{\|\phi\|_{L^2}^4}{\|\rho_0\|_{L^1}^2} \]
by Nash inequality in dimension two and conservation of the integral of $\rho.$ It follows that
\[ \|\phi(\cdot,t)\|_{L^2}^2 \leq {\rm min} \left(\|\rho_0\|_{L^2}^2, C\|\rho_0\|_{L^1}^2 t^{-1}\right). \]
By duality (and incompressibility of $u$), we also have
\[ \|\phi(\cdot,t)\|_{L^{\infty}} \leq {\rm min} \left(\|\rho_0\|_{L^{\infty}}, C\|\rho_0\|_{L^2} t^{-1/2}\right). \]
Combining these estimates via time split at $t/2$ yields
\[  \|\phi(\cdot,t)\|_{L^{\infty}} \leq {\rm min} \left(\|\rho_0\|_{L^{\infty}}, C\|\rho_0\|_{L^1} t^{-1}\right). \]
By comparison principle, the same bound applies to $\rho(x,t).$
Then we can estimate
\begin{equation}\label{rhoL21}
\|\rho(\cdot,t)\|_{L^2}^2 \leq \|\rho(\cdot,t)\|_{L^{\infty}}\|\rho(\cdot,t)\|_{L^1} \leq C\|\rho(\cdot,t/2)\|_{L^1} \|\rho(\cdot,t)\|_{L^1} t^{-1}.
\end{equation}
A strictly weaker bound
\[
\|\rho(\cdot,t)\|_{L^2}^2 \leq C\|\rho_0\|_{L^1} \|\rho(\cdot,t)\|_{L^1} t^{-1}
\]
 is also true. It implies that
\[ \partial_t \|\rho(\cdot,t)\|_{L^1} = -\epsilon \|\rho(\cdot,t)\|_{L^2}^2 \geq -\epsilon C\|\rho_0\|_{L^1} \|\rho(\cdot,t)\|_{L^1} t^{-1}, \]
and therefore
\[ \|\rho(\cdot,t/2)\|_{L^1} \leq 2\|\rho(\cdot,t)\|_{L^1}e^{\epsilon \|\rho_0\|_{L^1}}. \]
Substituting this into \eqref{rhoL21}, we get
\[  \|\rho(\cdot,t)\|_{L^2}^2  \leq C e^{\epsilon \|\rho_0\|_{L^1}} \|\rho(\cdot,t)\|_{L^1}^2 t^{-1}. \]
Hence, we have
\[ \partial_t \|\rho(\cdot,t)\|_{L^1} \geq -\epsilon {\rm min} \left( \|\rho_0\|_{L^2}^2, Ce^{\epsilon \|\rho_0\|_{L^1}} \|\rho(\cdot,t)\|_{L^1}^2 t^{-1}  \right). \]
Define $t_0$ by
\[
\|\rho_0\|_{L^2}^2 = Ce^{\epsilon \|\rho_0\|_{L^1}} \|\rho(\cdot,t_0)\|_{L^1}^2 t_0^{-1}.
\]
Then,  for $t\geq t_0,$ we have
\[ \partial_t \|\rho(\cdot,t)\|_{L^1} \geq -\epsilon C e^{\epsilon \|\rho_0\|_{L^1}} \|\rho(\cdot,t)\|_{L^1}^2 t^{-1}, \]
leading to
\[ \|\rho(\cdot,t)\|_{L^1} \geq \frac{\|\rho(\cdot,t_0)\|_{L^1}}{1+C\epsilon \|\rho(\cdot, t_0)\|_{L^1} \log (t/t_0) \exp(\epsilon\|\rho(\cdot,t_0)\|_{L^1})}. \]
If $\epsilon \|\rho_0\|_{L^1}\le c_0,$ with a sufficiently small $c_0$,
then
\[
t_0 \leq C\|\rho_0\|_{L^1}^2/\|\rho_0\|_{L^2}^2,
\]
and
\[
\|\rho(\cdot,t_0)\|_{L^1} \geq \|\rho_0\|_{L^1}(1-\epsilon C \|\rho_0\|_{L^1}).
\]
Hence, \eqref{simplerlb} follows.
\end{proof}

\section{The chemotactic case}\label{chemsection}

The proof of Theorem~\ref{chemore} is identical to the   case $q>2$
considered in \cite{KR10}, provided that we have global regularity of solutions. It will be convenient for us to work in a setting similar to \cite{KR10}.
Recall that we are considering the equation
\begin{equation}\label{chemnew}
\partial_t \rho+u\cdot\nabla\rho =
\Delta \rho+ \chi \nabla \cdot(\rho \nabla[(\Delta)^{-1}\rho])  -  \epsilon \rho^2, \,\,\,\rho(x,0)=\rho_0(x),~~x\in\Rm^2.
\end{equation}
Define
\[
\|f\|_{M_n} = \int_{\Rm^2} (|\rho(x)| + |\nabla \rho(x)|)(1+|x|^n)\,dx,
\]
and the Banach space $K_{s,n}$ is defined by the norm $\|f\|_{K_{s,n}} = \|f\|_{M_n} + \|f\|_{H^s}.$

\begin{theorem}\label{globreg}
Assume that $n>0$ and $s>d/2+1$ are integers and $\rho_0 \geq 0 \in K_{s,n}.$ Suppose that $u \in C^\infty(\Rm^d \times [0,\infty))$ is divergence free. Then there exists a unique solution $\rho(x,t)$ of the equation
\eqref{chemnew} in $C(K_{s,n}, [0,\infty)) \cap C^\infty(\Rm^d \times (0,\infty)).$
\end{theorem}
The local existence is proved in the same manner as in \cite{KR10}. Also, it was shown in \cite{KR10} that to prove global existence, it suffices to prove an a-priori global bound on Sobolev norms of the solution.
\begin{proof}
The first observation is that, integrating in space and time, we obtain a bound
\[
\epsilon \int_0^T \int_{\Rm^2} \rho(x,t)^2 \,dx \leq \|\rho_0\|_{L^1},
\]
for any smooth decaying solution of
\eqref{chemnew} on $[0,T] \times \Rm^2.$  This will be our main control which turns out to be sufficient for global regularity.
Multiply equation \eqref{chemnew} by $(-\Delta)^s \rho$ and integrate. Let us denote by $\dot{H}^s$ the homogenous Sobolev norm. We get 
\[ \frac12 \partial_t \|\rho\|_{\Dot{H}^s}^2 \leq \chi \left| \int_{\Rm^2} (\nabla \rho) \cdot (\nabla (-\Delta)^{-1}\rho) (-\Delta)^s \rho \,dx \right| +
\chi \left| \int_{\Rm^2} \rho^2 (-\Delta)^s \rho\,dx \right| + C(u)\|\rho\|^2_{H^s}
- \|\rho\|^2_{\Dot{H}^{s+1}}. \] 
Consider the expression
\[
\int_{\Rm^2} \rho^2 (-\Delta)^s \rho\,dx.
\]
Integrating by parts, we can represent this integral as a sum of terms of the form
\[
\int_{\Rm^2} D^l \rho D^{s-l}\rho D^s \rho\,dx,
\]
where $l=0,\dots,s$ and $D$ denotes any partial derivative. By H\"older inequality, we have
\[ \left| \int_{\Rm^2} D^l \rho D^{s-l}\rho D^s \rho\,dx \right| \leq \|\rho\|_{\Dot{H}^s}\|D^l \rho\|_{L^{p_l}}\|D^{s-l}\rho\|_{L^{q_l}}, \]
with any $p_l$ and $q_l$ satisfying $p_l^{-1}+q_l^{-1}=1/2.$ Observe that for $d=2$, every integer $0 \leq m \leq s$, and $2 \leq p<\infty$,
we have the Gagliardo-Nirenberg inequality
\[ \|D^m \rho\|_{L^p} \leq C \|\rho\|_{L^2}^{1-a} \|\rho\|^a_{\Dot{H}^{s+1}}, \,\,\,a= \frac{m-\frac{2}{p}+1}{s+1}. \]
Then
\[ \|D^l \rho\|_{L^{p_l}}\|D^{s-l}\rho\|_{L^{q_l}} \leq C\|\rho\|_{L^2}\|\rho\|_{\Dot{H}^{s+1}}, \]
and therefore
\[ \left| \int_{\Rm^2} \rho^2 (-\Delta)^s \rho\,dx \right| \leq C\|\rho\|_{L^2}\|\rho\|_{\Dot{H}^s}\|\rho\|_{\Dot{H}^{s+1}}. \]

Next, consider
\[
\int_{\Rm^2} (\nabla \rho) \cdot (\nabla (-\Delta)^{-1}\rho) (-\Delta)^s \rho \,dx.
\]
Integrating by parts, we get terms that can be estimated similarly to the previous case, using the fact that the double
Riesz transform $\partial_{ij}(-\Delta)^{-1}$ is bounded on $L^p,$ $1<p<\infty.$ The only exceptional terms that appear which have different structure are
\[ \int_{\Rm^2} (\partial_{i_1}\dots \partial_{i_s} \nabla \rho) \cdot  (\nabla (-\Delta)^{-1}\rho)\partial_{i_1}\dots \partial_{i_s} \rho \, dx \]
but these can be reduced to
\[
\int_{\Rm^2} (\partial_{i_1}\dots \partial_{i_s} \rho)^2 \rho \,dx
\]
by another integration by parts, and estimated as before. Altogether, we get
\begin{eqnarray*}
\frac12 \partial_t \|\rho\|_{\Dot{H}^s}^2 \leq C\xi \|\rho\|_{L^2}\|\rho\|_{\Dot{H}^s}\|\rho\|_{\Dot{H}^{s+1}} +C(u)\|\rho\|_{H^s}^2 - \|\rho\|^2_{\Dot{H}^{s+1}} \leq  \\
(C\xi^2 \|\rho\|_{L^2}^2 +C(u))\|\rho\|_{\Dot{H}^s}^2 + C(u)\|\rho\|_{L^2}^2 - \frac12 \|\rho\|^2_{\Dot{H}^{s+1}}.
\end{eqnarray*}
Notice that
\[ \|\rho\|_{\Dot{H}^s} \leq \|\rho\|^{\frac{s+d/2}{s+1+d/2}}_{\Dot{H}^{s+1}}\|\rho\|_{L^1}^{\frac{1}{s+1+d/2}}, \]
and so due to non-increase of $L^1$ norm,
\[ \frac12 \partial_t \|\rho\|_{\Dot{H}^s}^2 \leq (C\xi^2 \|\rho\|_{L^2}^2 +C(u))\|\rho\|_{\Dot{H}^s}^2 + C(u)\|\rho\|_{L^2}^2 -
c \|\rho\|_{\Dot{H}^s}^{2+\frac{2}{s+1+d/2}}. \] From this differential inequality and integrability of $\|\rho\|_{L^2}^2$ in time, a global bound for $\|\rho\|_{\Dot{H}^s}$ follows.
\end{proof}

\noindent {\bf Acknowledgement.} \rm AK is grateful to Stanford University for its hospitality
in winter 2011-2012. AK acknowledges support of the NSF grant DMS-1104415, and thanks the Institute for Mathematics and
Its Applications for stimulating atmosphere at the Workshop on Transport and Mixing in Complex and
Turbulent Flows in March 2010, where he learned about coral spawning from the talk of Prof. Jeffrey Weiss. LR is supported in part by the
NSF grant DMS-0908507.

\end{document}